\documentclass[12pt]{amsart}
\newif\ifdraft
\draftfalse

\ifdraft\usepackage[notref,notcite]{showkeys}\fi

\usepackage{indentfirst}
\usepackage{bm, mathrsfs, graphics,amssymb, amsmath, amsfonts, subeqnarray,setspace,graphicx,amsthm,epstopdf, mathtools, subfigure,color,mathrsfs}

-\marginparwidth \oddsidemargin -9mm \evensidemargin \oddsidemargin
\advance\hoffset by 5mm

\def\e{\epsilon}

\DeclareMathOperator{\var}{Var}

\newcommand\myeq{\mathrel{\stackrel{\makebox[0pt]{\tiny d}}{=}}}

\newtheorem{theorem}{Theorem}
\newtheorem{lemma}{Lemma}
\newtheorem{proposition}{Proposition}
\newtheorem{remark}{Remark}
\newtheorem{corollary}{Corollary}
\newtheorem{definition}{Definition}

\makeatletter
\@mparswitchfalse
\makeatother

\newcommand\blfootnote[1]{%
  \begingroup
  \renewcommand\thefootnote{}\footnote{#1}%
  \addtocounter{footnote}{-1}%
  \endgroup
}

\begin{document}

\title[FSDE satisfying fluctuation-dissipation theorem]{Fractional stochastic differential equations satisfying fluctuation-dissipation theorem}

\author{Lei Li}
\address{\hskip-\parindent
Lei Li\\
Department of Mathematics\\
Duke University\\
Durham, NC 27708, USA
}
\email{leili@math.duke.edu}

\author{Jian-Guo Liu}
\address{\hskip-\parindent
Jian-Guo Liu\\
Departments of Physics and Mathematics\\
Duke University\\
Durham, NC 27708, USA
}
\email{jliu@phy.duke.edu}

\author{Jianfeng Lu}
\address{\hskip-\parindent
Jianfeng Lu\\
Department of Mathematics, Department of Physics, and Department of Chemistry, \\
Duke University, Box 90320,\\
Durham, NC 27708, USA
}
\email{jianfeng@math.duke.edu}

\maketitle

\begin{abstract}
 We propose in this work a fractional stochastic differential equation (FSDE) model consistent with the over-damped limit of the generalized Langevin equation model. As a result of the `fluctuation-dissipation theorem', the differential equations driven by fractional Brownian noise to model memory effects should be paired with Caputo derivatives, and this FSDE model should be understood in an integral form. We establish the existence of  strong solutions for such equations and discuss the ergodicity and convergence to Gibbs measure. In the linear forcing regime, we show rigorously the algebraic convergence to Gibbs measure when the `fluctuation-dissipation theorem' is satisfied, and this verifies that satisfying `fluctuation-dissipation theorem' indeed leads to the correct physical  behavior. We further discuss possible approaches to analyze the ergodicity and convergence to Gibbs measure in the nonlinear forcing regime, while leave the rigorous analysis for future works. The FSDE model proposed is suitable for systems in contact with heat bath with power-law kernel and subdiffusion behaviors.
\end{abstract}


\section{Introduction}

For a particle in contact with a heat bath (such as a heavy particle
surrounded by light particles), the following stochastic equation is
often used to describe the evolution of the velocity of the particle
\[ 
m\dot{v}=-\gamma v+\eta,
\]
where dot denotes derivative on time, $-\gamma v$ counts for friction and $\eta$ is a Gaussian white noise which could be understood as the distributional derivative of the Brownian motion (or Wiener process) up to a constant factor. This equation should be understood in the SDE form 
\[ 
m\,dv=-\gamma v\,dt+\sqrt{2D_x}\,dW,
\]
where $W$ is a standard Brownian motion and $D_x$ is some constant to be determined.  Adding the equation for position and considering external force, one has the Langevin equation: 
\begin{gather}\label{eq:langevin}
\dot{x}=v, ~~ m\dot{v}=-\nabla V(x)-\gamma v+\eta.
\end{gather}
Since the friction coefficient $\gamma$ and random force $\eta$ both stem from interactions between the particle and the environment, they should be related. The `fluctuation-dissipation theorem' \footnote{Note that we are putting quotes for the physical theorems as they are critical claims from physics compared with mathematical theorems that are rigorously justified.} (\cite{nyquist28,cw51}) provides a precise
connection between them, such that the covariance satisfies
\begin{gather}\label{eq:flucdiss1}
\mathbb{E}(\eta(t_1)\eta(t_2))=2kT\gamma \delta(t_1-t_2),
\end{gather}
where $k$ is the Boltzmann constant and $T$ is the absolute
temperature, leading to $D_x=kT\gamma$. $\mathbb{E}$ is the `ensemble average' in physical language and it is `expectation' over some underlying probability space in mathematical language.  Relation \eqref{eq:flucdiss1} was formulated by Nyquist in \cite{nyquist28} and then justified by Callen and Welton in \cite{cw51}. The physical meaning of this relation is that the fluctuating forces must restore the energy dissipated by the friction so that the balance is achieved and the temperature of the heavy particle can reach the correct value. To see this in another view point, one may derive, either using Ito's formula or using Green-Kubo formula, that $D_x$ is actually the diffusion constant for position $x$, and $D_x=kT\gamma$ is called the Einstein-Smoluchowski relation \cite{mprv08}. This relation also says that the fluctuation and dissipation must be related.

In the `overdamped' regime where the inertia can be neglected ($m\ll 1$), the Langevin equation is reduced to the following well-known SDE \cite{freidlin04}: 
\begin{gather}
\gamma\,dx=-\nabla V(x)\,dt+\sqrt{2D_x}\,dW.
\end{gather} 

In \cite{mori65,kubo66},  the generalized Langevin equation (GLE) was proposed to model particle motion in contact with a heat bath when the random force is no longer memoryless: 
\begin{gather}\label{eq:gneralizedLangevin}
  \dot{x}=v,~~ m\dot{v}=-\nabla V-\int_{t_0}^t\gamma(t-s)v(s)\,ds
  +R(t),
\end{gather}
where $R(t)$ is some random force. Now the friction is the convolution between a kernel function $\gamma$  and the velocity $v(s)$ so that there is memory in dissipation in this model. For the particle to achieve equilibrium at the prescribed temperature, the fluctuating force $R(t)$ and the friction kernel $\gamma$ must be related. Without the external force (i.e.
$\nabla V=0$), Kubo assumed that $\mathbb{E}(v(t_0)R(t))=0, t>t_0$ and that $v$ is a stationary process. He derived formally (though he used the existence of the one-sided Fourier transform of $\gamma$, the formal derivation still holds if $\gamma\notin L^1[0,\infty)$ as we can understand the transform in the distribution sense or replace the one-sided Fourier transform with Laplace transform) that
\begin{gather}\label{eq:flucdis}
\mathbb{E}( R(t_0)R(t_0+t))=m\mathbb{E}(v(t_0)^2) \gamma(|t|)=kT\gamma(|t|).
\end{gather}
There are other formal derivations as well (e.g. \cite{felderhof78}). These derivations are not fully convincing though on the mathematical rigorous level. In \cite{kubo66}, Kubo assumed the relation $\mathbb{E}(v(t_0)R(t))=0, t>t_0$ arguing using causality. The issue is though $R(t)$ does not affect $v(t_0)$, $v(t_0)$ can affect $R(t)$. In \cite{felderhof78},
Felderholf obtained this relation from `Nyquist's theorem', while no justification is given to the latter. 

For a more convincing and rigorous derivation of the GLE \eqref{eq:gneralizedLangevin} and relation \eqref{eq:flucdis}, one
could start from a system of interacting particles as the Kac-Zwanzig model (see \cite{fkm65,zwanzig73,gks04,kou08}).  In this model, the surrounding particles in the heat bath have harmonic interactions with the particle under consideration, which is a good approximation if the configuration is near equilibrium.  The whole system evolves under the total Hamiltonian. If the initial data satisfy the Gibbs measure, then after integrating out the variables for the surrounding particles, one obtains the GLE where the relation \eqref{eq:flucdis} is satisfied. From the Kac-Zwanzig model, we may find that in GLE the random force $R(t)$ is not necessarily independent of $x(0)$. 

Relation \eqref{eq:flucdis} is called the `fluctuation-dissipation theorem' for GLE. This relation simply says the random force must
balance the friction so that the system has a nontrivial equilibrium corresponds to the prescribed temperature.  Note that if the kernel
$\gamma(t)$ tends to $\gamma \delta(t)$, the relation \eqref{eq:flucdiss1} can be recovered. The coefficient `2' comes from the fact that
\[
 \int_{-\infty}^{\infty}\mathbb{E}(R(t_0)R(t_0+t))\,dt=2kT\int_0^{\infty}\gamma(t)\,dt.
\]
There are few rigorous mathematical justifications of the `fluctuation-dissipation theorem', all in the context of generalized
Langevin equations. In \cite{pavliotis14}, the author tried to rephrase the `fluctuation-dissipation theorems' and the related linear
response theory in mathematical language.  Hairer and Majda in \cite{hm10} developed a framework to justify the use the linear
response theory through the `fluctuation-dissipation theorem' for studying climate models.

\smallskip

In this work, through a scaling argument, we find it reasonable to consider the over-damped limit of GLE driven by fractional Brownian noise, and obtain the following class of fractional SDE (FSDE) models (Equation \eqref{eq:fsde1})
\begin{gather*}
D_c^{\alpha}x=-V'(x)+C_H\dot{B}_H, 
\end{gather*}
where the fractional derivative is in Caputo sense while $\dot{B}_H$ is the fractional Brownian noise (the distributional derivative of fractional Brownian motion). This differential form can be rewritten as an integral form (Equation \eqref{eq:fsde2}):
\begin{gather*}
x(t)=x(0)-\frac{1}{\Gamma(\alpha)}\int_0^t(t-s)^{\alpha-1}V'(x(s))\,ds+\frac{C_H}{\Gamma(\alpha)}\int_0^t(t-s)^{\alpha-1}dB_{H},
\end{gather*}
which is viewed as the rigorous definition of our FSDE model. After
proving that the stochastic integral is a continuous process in
Section \ref{sec:exisuniq}, the existence and uniqueness of strong
solutions become straightforward.

Let us remark that using fractional Brownian motion as a model
  for long range correlations is quite common: for example, waves in
  random media \cite{ms11}, subdiffusion process in complex system
  \cite{kouxie04,mlckx05,mwbk09,kou08,db09}. It has been observed in \cite{kouxie04,mlckx05,kou08} that the systems of protein molecules have power-law memory kernel and subdiffusion behavior.  Remarkably,  Kou and Xie \cite{kouxie04,kou08} showed that incorporating  fractional Brownian noise into the generalized Langevin equation  yields a model with a power-law kernel for subdiffusion and the  results had excellent agreement with a single-molecule experiment from biological science.

If $\alpha=\alpha^*:=2-2H$, the `fluctuation-dissipation
  theorem' is satisfied. When the force is linear, we show rigorously that
  the process has ergodicity and converges algebraically to the Gibbs measure (see Theorem \ref{thm:linear}).
  When the force is nonlinear, studying the ergodicity and asymptotic behavior is 
  challenging. We believe this problem must be solved by rewriting the FSDE model into Markovian processes. 
  So, we propose two possible approaches. The first approach is to rewrite 
  our FSDE model as an infinite-dimensional Ornstein-Uhlenbeck (OU) process with mixing. 
  We hope this infinite-dimensional OU with mixing can be a possible framework for proving the convergence to equilibrium satisfying Gibbs measure.  Another approach is to take the limit in a heat bath model.
 In summary, satisfying the `fluctuation-dissipation theorem' leads to the correct physical
  behavior: there is balance between the dissipation and fluctuation
  effects from the random forcing such that the Gibbs measure is the
  final equilibrium distribution. This means that in the  correct physical FSDE models fractional Brownian noise must be
  paired with Caputo derivatives.
  
While FSDEs have been discussed in
some previous works already, our FSDE model \eqref{eq:fsde2} motivated
by the `fluctuation-dissipation theorem' seems to be new. The authors
of \cite{no02,hairer05,hp11} discussed FSDEs driven by fractional
Brownian motions but used the usual first order derivative, which
means that the convolution kernel for friction is a Dirac delta
  and there is no memory in the dissipating term, while the
  fluctuation term is given by fractional Brownian motions that have
  memory so that there is no balance. In \cite{srr13}, the Caputo
derivative is used but they used the usual white noise to drive the
process. According to the above formal derivation, when modeling a
particle in contact with a heat bath with memory effects, the natural
noise associated with the Caputo derivative should be the fractional
noise. This means we will probably require $\alpha=\alpha^*$ for the
correct model from physical concerns. We admit however that it is
possible that the models with $\alpha\neq \alpha^*$ may be used to
describe some other situations instead of the physical case we
consider here.

The rest of the paper is organized as follows. In Section \ref{sec:processg}, we study the stochastic integral in our FSDE model \eqref{eq:fsde2} in detail and prove that it is continuous. Using the continuity of the stochastic integral, we obtain in Section \ref{sec:exisuniq} the existence and uniqueness of strong solutions for FSDE \eqref{eq:fsde2} on the interval $[0,\infty)$ provided $V'(\cdot)$ is Lipschtiz continuous. In Section \ref{sec:asym}, we focus on the asymptotic behavior of the strong solutions of \eqref{eq:fsde2}. In particular, in the linear regimes, (i.e. $V'(\cdot)$ is a linear function), we compute the solutions exactly and show that the solution converges in distribution to a stationary process satisfying Gibbs measure. In the nonlinear regime, we provide two possible frameworks for studying the asymptotic behaviors when the `fluctuation-dissipation theorem' is satisfied. We argue formally that the FSDE can be reduced from some Markovian processes in infinite dimensions. The rigorous study of the nonlinear regimes is left for future works.

\section{The FSDE model}

In this section, we propose the fractional SDE model from the GLE with fractional Brownian noise. By a scaling argument in GLE, we argue that in the regimes where the environment is viscous or the mass is small, we can consider the over-damped limit of GLE driven by fractional Brownian noise (i.e., the (distributional) derivative of fractional Brownian motion) and obtain the fractional SDE model, in which the Caputo derivative is associated with the fractional Brownian noise. This model is new. It recovers the subdiffusion discussed in \cite{kouxie04,mwbk09} and satisfies the `fluctuation-dissipation theorem'.

\subsection{Fractional Brownian noise in complex systems}

The studies in \cite{kouxie04,mlckx05,kou08,mwbk09} indicate that
  fractional Brownian noise commonly arises in complex physical
  systems. We now give a brief introduction to fractional Brownian
  motion and present the GLE with fractional Brownian noise.

The fractional Brownian motion $B_H$ (see \cite{mv68,nualart06} for more detailed discussions) with Hurst parameter $H\in(0,1)$ is a Gaussian process (i.e., the joint distribution for $(B_H(t_1), \ldots, B_H(t_d))$ is a $d$-dimensional normal distribution for any $(t_1, \ldots, t_d)\in\mathbb{R}_+^d$) defined on some probability space $(\Omega, \mathcal{F}, P)$ with mean zero and covariance 
\begin{gather}\label{fb:cov}
\mathbb{E}(B_t^HB_s^H)=R_H(s, t)=\frac{1}{2}\left(s^{2H}+t^{2H}-|t-s|^{2H}\right),
\end{gather}
where $\mathbb{E}$ means the expectation over the underlying
probability space. By definition, $B_H$ has stationary increments
which are normal distributions with
$\mathbb{E}((B_H(t)-B_H(s))^2)=(t-s)^{2H}$. By the Kolmogorov
continuity theorem, $B_H$ is H\"older continuous with order
$H-\epsilon$ for any $\epsilon\in (0, H)$. $B_H$ has finite
$1/H$-variation. Besides, it is self similar:
$B_H(t)\myeq a^{-H}B_H(at)$ where `$\myeq$' means they have the same distribution. It is non-Markovian except for $H=1/2$
when it is reduced to the Brownian motion (i.e., Wiener process).

The existence of fractional Brownian motion can be proved by some explicit representations. In \cite{mv68}, the following representation is given 
\begin{multline}
B_H(t)=C_1(H)\left(\int_0^t(t-s)^{H-\frac{1}{2}}dW(s)+\int_{-\infty}^0((t-s)^{H-\frac{1}{2}}-(-s)^{H-\frac{1}{2}})dW(s)\right)\\
=C_1(H)\int_{-\infty}^0(-r)^{H-\frac{1}{2}}(dW(r+t)-dW(r)),
\end{multline}
where $W$ is a normal Brownian motion and $C_1(H)$ is a constant to make \eqref{fb:cov} valid. This is also used in \cite{hairer05}. In \cite{du99,no02},  one uses
\begin{gather}
B_H(t)=C_2(H)\int_0^t(t-s)^{H-\frac{1}{2}}F\left(H-\frac{1}{2}, \frac{1}{2}-H, H+\frac{1}{2}, 1-\frac{t}{s}\right)dW(s),
\end{gather}
where $F$ is the Gauss hypergeometric function. Another representation in \cite{pt01} using fractional integrals might be useful sometimes, which we choose to omit here. 

One can show that $(B_H(t+h)-B_H(t))/h$ converges in distribution (i.e. under the topology of the dual of $C_c^{\infty}(0,\infty)$) to $\dot{B}_H(t)$ where the dot represents distributional time derivative. We check that
\begin{multline}
  \lim_{h\to 0^+, h_1\to 0}\mathbb{E}\left(\frac{B_H(h)}{h}\frac{B_H(t+h_1)-B_H(t)}{h_1}\right)\\
  =\lim_{h\to 0^+, h_1\to 0}
  \frac{1}{2hh_1}\bigl((t+h_1)^{2H}-(t+h_1-h)^{2H}-t^{2H}+(t-h)^{2H}\bigr) \\
  =H(2H-1) t^{2H-2}.
\end{multline}

If we pick the initial time in \eqref{eq:gneralizedLangevin} as $t_0=0$ 
and consider the random noises corresponding to fractional Brownian motion as discussed:
\begin{gather}
R_H(t)=\frac{\sqrt{kT\gamma_0}}{\sqrt{H(2H-1)\Gamma(2H-1)}} \dot{B}_H(t),
\end{gather} 
where $\gamma_0$ is a constant representing the typical scale of friction, we then have the GLE model 
\begin{gather}\label{eq:inertial}
m\dot{v}=-\nabla V(x)-\frac{\gamma_0}{\Gamma(2H-1)}\int_0^t(t-s)^{2H-2}v(s)\,ds+R_H(t)
\end{gather}
following the `fluctuation-dissipation theorem'. 

We will assume throughout the paper that
\begin{gather}
H\in \left(\frac{1}{2}, 1\right),
\end{gather}  
as they are the physically most realistic regimes \cite{kou08} and consequently $2-2H\in (0,1)$. 

\subsection{Over-damped limit and the FSDE model}

Assume that we consider the fractional diffusion regime with time scale $T_{t}$,  the length scale $L=\sqrt{kT/\gamma_0}T_t^{1-H}$, and velocity scale $L/T_t=\sqrt{kT/\gamma_0}T_t^{-H}$. We then scale the energy with $kT$, fractional Brownian motion with $T_t^{H}$ and scale the noise with $\sqrt{kT\gamma_0}T_t^{H-1}$. The dimensionless GLE reads
\begin{gather*}
\frac{mT_t^{2H}}{\gamma_0}\dot{v}=-\nabla V-\frac{1}{\Gamma(2H-1)}
\int_0^t(t-s)^{2H-2}v(s)ds+R_H(t)
\end{gather*}

In the regimes where $mT_t^{2H}/\gamma_0$ is small (viscous environment, particle is small etc), the $m\dot{v}$ term in \eqref{eq:inertial} can be neglected, and we have the following dimensionless over-damped equation with fractional noise:
\begin{gather}\label{eq:formalfsde}
\frac{1}{\Gamma(2H-1)}\int_0^t(t-s)^{2H-2}v(s)\,ds=-\nabla V(x)+R_H(t).
\end{gather}

In the following discussion, we will always assume the equations are dimensionless  while the variables $k$ and $T$ might be used to denote other quantities.

Recall that the Caputo derivative (\cite{gm97,kst06}) starting from $t=0$ for a $C^1$
function is given by
\begin{gather}\label{eq:caputo}
D_c^{\alpha}w=\frac{1}{\Gamma(1-\alpha)}\int_0^t\frac{\dot{w}(s)}{(t-s)^{\alpha}}\,ds.
\end{gather}
Note that $v(s)=\dot{x}(s)$. The left hand side of Equation
\eqref{eq:formalfsde} formally becomes the Caputo derivative of $x$
with $\alpha=2-2H$ and the equation becomes a fractional SDE:
\begin{gather}
D_c^{2-2H}x=-\nabla V(x)+R_H(t).
\end{gather}
This means that the power-law memory kernel yields the Caputo derivative of the trajectory naturally. This over-damped fractional SDE model is simpler compared with the GLE model \eqref{eq:inertial} and we expect it to contain the essential physics (the subdiffusion and `fluctuation-dissipation theorem') as we will study. 

From here on, we will only consider 1D case ($x\in \mathbb{R}$) for
convenience while the general dimension is similar.  The above
discussion then motivates us to consider the fractional stochastic
differential equation (FSDE) where we relax the constraint between $H$
and $\alpha$:
\begin{gather}\label{eq:fsde1}
D_c^{\alpha}x=-V'(x)+C_H\dot{B}_H, 
\end{gather}
where
\begin{gather}
C_H=\frac{1}{\sqrt{H(2H-1)\Gamma(1-\alpha)}}
\end{gather}
for $\alpha\in (1-H, 1]$. The index obtained from the `fluctuation-dissipation theorem' is denoted as $\alpha^*=2-2H$. We will also denote the (one-sided) kernel associated with the Caputo derivative as 
\begin{gather}
\gamma(t)=\frac{\theta(t)}{\Gamma(1-\alpha)}t^{-\alpha},
\end{gather}
where $\theta(t)$ is the standard Heaviside step function.

In \cite{liliu_prep}, a definition of the Caputo derivative based on a convolution group was proposed, which agrees with \eqref{eq:caputo} when the function is absolutely continuous on $(0,t)$. The observation of the underlying convolution group structure allows us to de-convolve and change the Caputo derivative to integral form as
\begin{multline}\label{eq:fsde2}
x(t)=x(0)+\frac{1}{\Gamma(\alpha)}\int_0^t(t-s)^{\gamma-1} D_c^{\gamma}x (s)\,ds\\
=x(0)-\frac{1}{\Gamma(\alpha)}\int_0^t(t-s)^{\alpha-1}V'(x(s))\,ds+\frac{C_H}{\Gamma(\alpha)}\int_0^t(t-s)^{\alpha-1}dB_{H},
\end{multline}
where we formally used $R_H\,ds=C_H\dot{B}_H \,ds=C_H\,dB_H$. This integral will then be understood as the rigorous definition of the FSDE \eqref{eq:fsde1}. The last term in \eqref{eq:fsde2} is an integral with respect to fractional Brownian motion, which we will make the meaning precise later.  We will study FSDE  \eqref{eq:fsde2} and try to understand the role of  the `fluctuation-dissipation theorem'.  For convenience, we denote
\begin{gather}\label{eq:proG}
G(t)=\frac{C_H}{\Gamma(\alpha)}\int_0^t(t-s)^{\alpha-1}dB_H(s)=\int_0^{\infty}f_t(s)\,dB_H(s) ,
\end{gather}
where $f_t(s)=\frac{C_H}{\Gamma(\alpha)}((t-s)^+)^{\alpha-1}$ and
$\alpha\in (1-H, 1)$. We shall study the process $G$ in Section \ref{sec:processg}.

\section{The process $G$ as a stochastic integral}\label{sec:processg}

To make the meaning of the FSDE precise, we must understand the
process $G$. In this section, we first review the stochastic integrals
with respect to fractional Brownian motions and then study some
properties of $G$.

\subsection{Stochastical integrals driven by fractional Brownian motions}

The stochastic integrals with respect to fractional Brownian motions have been thoroughly discussed in literature \cite{zahle98,mn00,du99,dhp00}. In \cite{zahle98,mn00}, the stochastic integrals are defined pathwise using the Riemann-Stieltjes integrals by making use of certain properties of the paths.  In \cite{du99,dhp00}, the so-called Malliavin calculus is used to define the stochastic integrals (Wick-Ito-Skorohod integrals, or the `divergence') and the Ito's formula is established, which connects both definitions. For a review, one can refer to \cite{nualart06,bhoz08}.  
In the case that the integrand is deterministic, those two definitions agree. By \eqref{eq:fsde2}, we only need the integrals of deterministic processes with respect to fractional Brownian motion. We shall give a brief introduction to the theory for deterministic processes and the readers can turn to the references listed here for general processes. 

Let us fix $T>0$ and define the stochastic integrals on the interval $[0, T]$. The definition of integration of deterministic processes on $[0, T]$ starts with the step functions. Let $\mathscr{E}$ be the set of all step functions on $[0, T]$, i.e. 
$\varphi\in\mathscr{E}$ is given by 
\begin{gather}
\varphi=\sum_{j=1}^m a_j 1_{(t_{j-1}, t_j]}(t),
\end{gather}
where $1_{E}(t)$ is the characteristic function of set $E$. The integral $B^H(\varphi)$ is defined by 
\begin{gather}
B^H(\varphi)=\int_0^T \varphi \, dB_H(t)=\sum_{j=1}^m a_j\Big(B_H(t_j)-B_H(t_{j-1})\Big).
\end{gather}
Consider the inner product
\begin{gather}\label{eq:inner1}
\langle \varphi_1,\varphi_2\rangle_{\mathscr{H}}=\mathbb{E}(B^H(\varphi_1)B^H(\varphi_2)).
\end{gather}
It is easily verified that $\forall \varphi_1, \varphi_2\in\mathscr{E}$, 
\begin{multline}\label{eq:inner}
\langle \varphi_1,\varphi_2\rangle_{\mathscr{H}}=H(2H-1)\int_0^T\int_0^T |r-u|^{2H-2}\varphi_1(r)\varphi_2(u)\,du dr\\
=\frac{\pi \kappa(2\kappa+1)}{\Gamma(1-2\kappa)\sin(\pi\kappa)}\int_0^Ts^{-2\kappa}(I^{\kappa}u^{\kappa}f)(s)(I^{\kappa}u^{\kappa}g)(s)\,ds,
\end{multline}
where $\kappa=H-\frac{1}{2}$ and $I^{\kappa}$ is the right Riemann-Liouville fractional calculus, given by (\cite{pt01}):
\begin{equation*}
(I^{\kappa}f)(s)=
\begin{cases} \displaystyle
\frac{1}{\Gamma(\kappa)}\int_s^{T}f(u)(u-s)^{\kappa-1}du, & \kappa>0, \\
\displaystyle
-\frac{1}{\Gamma(1-\kappa)}\frac{d}{ds}\int_s^Tf(u)(u-s)^{-\kappa}du. & \kappa<0.
\end{cases}
\end{equation*}

This then motivates the definition of
\begin{gather}
\mathscr{H}_0=\Big\{\varphi\in L_{loc}^1[0, T]: \int_0^T\int_0^T|r-u|^{2H-2}|\varphi(r)||\varphi(u)|dr du<\infty\Big\}
\end{gather}
and 
\begin{gather}
\Lambda=\left\{f\in L_{loc}^1[0,T]: \int_0^Ts^{-2\kappa}(I^{\kappa}u^{\kappa}f)^2(s)\,ds<\infty\right\}.
\end{gather}
Clearly, $\mathscr{H}_0\subset\Lambda$. The integral $B^H(\varphi)$ can then be defined for $\varphi\in \Lambda$ by approximating them with step functions. In \cite{pt00,pt01}, it is shown that both inner product spaces $\mathscr{H}_0$ and $\Lambda$ are not complete and therefore not Hilbert spaces. However, the space $B^H(\mathscr{E})$ clearly has a closure in $L^2(\Omega, P)$. This means some elements in the closure corresponds to distributions that are not in $L^1_{loc}[0,T]$. Let $\mathscr{H}$ be the space of the closure of $\mathscr{E}$ under the inner product \eqref{eq:inner1} and thus $\mathscr{H}$ contains some distributions. $\forall \varphi_1,\varphi_2\in \mathscr{H}_0\subset \mathscr{H}$,
\begin{gather}\label{eq:innerh}
\langle \varphi_1,\varphi_2\rangle_{\mathscr{H}}=\mathbb{E}(B^H(\varphi_1)B^H(\varphi_2))=H(2H-1)\int_0^T\int_0^T |r-u|^{2H-2}\varphi_1(r)\varphi_2(u)\,du dr.
\end{gather}
The following lemma provides a convenient way to check that some deterministic processes can be integrated by fractional Brownian motion (\cite{mmv01,nualart06}):
\begin{lemma}\label{lmm:embeddingH}
If $H>1/2$ and $\varphi\in L^{1/H}([0, T])$, then there exists $b_H>0$ such that
\begin{gather}
\|\varphi\|_{\mathscr{H}_0}\le b_H\|\varphi\|_{L^{1/H}[0, T]}.
\end{gather}
where 
\[
\|\varphi\|_{\mathscr{H}_0}^2=\int_0^T\int_0^T|r-u|^{2H-2}|\varphi(r)||\varphi(u)|\,drdr.
\]
\end{lemma}

\subsection{Some basic properties of $G$}
We can easily verify that $f_t\in L^{1/H}[0, T]$ whenever $t\le T$, and hence the integral on $[0, T]$ is well defined. Further, for any $T_1>t, T_2>t$, the integral of $f_t$ over $[0, T_1]$ and $[0, T_2]$ agree on $[0, \min(T_1, T_2)]$. In this sense, the integral $\int_0^{\infty}f_t(s)dB_H(s)$ can then be understood as in $[0, T]$ for any $T>t$.

Roughly speaking, since $B_H$ is $H-\epsilon$ H\"older continuous for any $\epsilon\in (0, H)$, $G(t)$ should be like $\alpha+H-1-\epsilon$ H\"older continuous for any $\epsilon\in (0, \alpha+H-1)$ by the regularity of $B_H$. 
We shall make this precise in this subsection. 
\begin{lemma}
$G(t)$ is a Gaussian process with mean zero and covariance given by 
\begin{multline}
\phi(t_1, t_2)=\mathbb{E}(G(t_1)G(t_2))=\frac{B(2H-1,\alpha)}{B(\alpha,1-\alpha)\Gamma(\alpha)} \times \\
 \int_0^{\min(t_1, t_2)}dr\Big((t_1-r)^{\alpha-1}(t_2-r)^{2H-2+\alpha}
+(t_2-r)^{\alpha-1}(t_1-r)^{2H-2+\alpha}\Big).
\end{multline}
In particular, if $\alpha=\alpha^*$, $G(t)\myeq \beta_H B_{1-H}$ where $\myeq$ means they have the same distribution, and 
\begin{gather}
\beta_H=\frac{\sqrt{2}}{\sqrt{\Gamma(3-2H)}}.
\end{gather}
In other words, $G(t)$ is a fractional Brownian motion with Hurst parameter $1-H$ up to a constant $\beta_{H}$ if $\alpha=\alpha^*$.
\end{lemma}

\begin{proof}
Clearly, $G(t)$ is a Gaussian process with mean zero because any linear operation of Gaussian process is again Gaussian. 

Without loss of generality, we can assume $t_2\ge t_1\ge 0$. The covariance can be computed using the isometry \eqref{eq:innerh}
\begin{multline*}
\mathbb{E}(G(t_1)G(t_2))=\langle f_{t_1}, f_{t_2}\rangle_{\mathscr{H}}=\\
\frac{1}{\Gamma(\alpha)B(\alpha,1-\alpha)}\int_0^{t_1}\int_0^{t_2}|r-u|^{2H-2}(t_1-r)^{\alpha-1}(t_2-u)^{\alpha-1}dudr.
\end{multline*}
We break the integral into two parts $I_1+I_2$, where 
\[
I_1=\frac{1}{\Gamma(\alpha)B(\alpha,1-\alpha)}\iint_{u\ge r}\ldots dudr, ~I_2=\frac{1}{\Gamma(\alpha)B(\alpha,1-\alpha)}\iint_{r\ge u}\ldots dudr.
\] 
By explicit computation,
\begin{multline*}
I_1=\frac{1}{\Gamma(\alpha)B(\alpha,1-\alpha)}\int_0^{t_1}dr(t_1-r)^{\alpha-1}\int_r^{t_2} du(u-r)^{2H-2}(t_2-u)^{\alpha-1}\\
=\frac{B(2H-1, \alpha)}{\Gamma(\alpha)B(\alpha,1-\alpha)}\int_0^{t_1} (t_1-r)^{\alpha-1}(t_2-r)^{2H-2+\alpha}dr.
\end{multline*}
This can further be written in terms of the so-called hypergeometric functions but we choose not to do it.
Similarly, 
\begin{multline*}
I_2=\frac{1}{\Gamma(\alpha)B(\alpha,1-\alpha)}\int_0^{t_1}du(t_2-u)^{\alpha-1}
\int_u^{t_1}dr(r-u)^{2H-2}(t_1-r)^{\alpha-1}\\
=\frac{B(2H-1, \alpha)}{\Gamma(\alpha)B(\alpha,1-\alpha)}\int_0^{t_1}(t_2-u)^{\alpha-1}(t_1-u)^{2H-2+\alpha} du.
\end{multline*}
If $\alpha=\alpha^*=2-2H$, the integrals $I_1$ and $I_2$ can be evaluated exactly:
\begin{gather*}
I_1+I_2=\frac{1}{\Gamma(3-2H)}\left(t_1^{2-2H}+t_2^{2-2H}-(t_2-t_1)^{2-2H}\right),
\end{gather*}
which shows the last claim.
\end{proof}

The above computation shows trivially that 
\begin{corollary}
In the case $V(x)$ is a constant, the solution of FSDE $D_c^{2-2H}x=R_H(t)$ satisfies $\var(x(t))\propto t^{2-2H}$. In other words, we have subdiffusion.
\end{corollary}

\begin{remark}
This agrees with the Langevin model in \cite[Theorem 2.2]{kou08}, though the author was discussing the case with mass.
Further, this corollary shows that the solution to our model usually is a fractional Brownian motion with a Hurst parameter $1-H\in (0, 1/2)$, and this agrees with the data analysis in \cite{kouxie04,mlckx05,mwbk09}, implying that our model makes physical sense.
\end{remark}

\begin{proposition}\label{pro:contofg}
There exists $C>0$ such that 
$\mathbb{E}|G(t_2)-G(t_1)|^2\le C|t_2-t_1|^{2H+2\alpha-2}$ and therefore $G(t)$ is $H+\alpha-1-\epsilon$ H\"older continuous for any $\epsilon>0$.
\end{proposition}
\begin{proof}
\[
\mathbb{E}|G(t_2)-G(t_1)|^2=\phi(t_2,t_2)+\phi(t_1, t_1)-2\phi(t_1, t_2).
\]
To be notationally convenient, let us define 
\[
\varphi(s, t)=\frac{B(\alpha,1-\alpha)\Gamma(\alpha)}{B(2H-1,\alpha)}\phi(s,t).
\]
Without loss of generality, we assume $t_2\ge t_1$. Applying $a+b\ge 2\sqrt{ab}$ whenever $a\ge 0, b\ge0$, we have
\[
\varphi(t_1, t_2)\ge 2\int_0^{t_1}(t_2-r)^{H+\alpha-3/2}(t_1-r)^{H+\alpha-3/2}dr
\]
If $H+\alpha-3/2\le 0$, then, 
\[
\varphi(t_1, t_2)\ge
2\int_0^{t_1}(t_2-r)^{2H+2\alpha-3}dr
=\frac{2}{2H+2\alpha-2}(t_2^{2H+2\alpha-2}-(t_2-t_1)^{2H+2\alpha-2}).
\]
Hence, 
\begin{multline*}
\mathbb{E}|G(t_2)-G(t_1)|^2\le C_1\Big(t_1^{2H+2\alpha-2}-t_2^{2H+2\alpha-2} \\
+2(t_2-t_1)^{2H+2\alpha-2}\Big)\le 2C_1(t_2-t_1)^{2H+2\alpha-2},
\end{multline*}
since $0<2H+2\alpha-2\le 1$, $t_1^{2H+2\alpha-2}-t_2^{2H+2\alpha-2}\le 0$.

If $H+\alpha-3/2> 0$, then 
\begin{multline*}
\varphi(t_2, t_2)+\varphi(t_1, t_1)-2\varphi(t_1, t_2)
=\int_{t_1}^{t_2}(t_2-r)^{2H+2\alpha-3}dr+\\
+\int_0^{t_1}((t_2-r)^{H+\alpha-3/2}-(t_1-r)^{H+\alpha-3/2})^2 dr.
\end{multline*}
The first integral is easily seen to be bounded by $C|t_2-t_1|^{2H+2\alpha-2}$ for some constant $C$. For the second term, we have:
\begin{gather*}
\left((t_2-r)^{H+\alpha-\frac{3}{2}}-(t_1-r)^{H+\alpha-\frac{3}{2}}\right)^2
=\left(H+\alpha-\frac{3}{2}\right)^2\left(\int_{t_1}^{t_2}(s-r)^{H+\alpha-\frac{5}{2}}ds\right)^2.
\end{gather*}
Let $I_{\epsilon}=(\int_{t_1}^{t_2}(s-r+\epsilon)^{H+\alpha-5/2}ds)^2$ with $r\le t_1$. Then, 
\begin{multline*}
\int_0^{t_1}I_{\epsilon}dr
\le (t_2-t_1)\int_0^{t_1}\int_{t_1}^{t_2}(s-r+\epsilon)^{2H+2\alpha-5}dsdr=(t_2-t_1)\int_{t_1}^{t_2}\int_0^{t_1}\ldots drds\\
=\frac{(t_2-t_1)}{|2H+2\alpha-4|(2H+2\alpha-3)}\Big((t_2-t_1+\epsilon)^{2H+2\alpha-3}-\epsilon^{2H+2\alpha-3}\\
-(s+\epsilon)^{2H+2\alpha-3}|_{t_1}^{t_2}\Big) \le C_{\alpha,H}(t_2-t_1)(t_2-t_1+\epsilon)^{2H+2\alpha-3}.
\end{multline*}
Note that $2H+2\alpha-4<0$. Taking $\epsilon\to 0$ shows that the second term is bounded by $C(t_2-t_1)^{2H+2\alpha-2}$. 

The Kolmogorov continuity criteria shows that $G(t)$ is $H+\alpha-1-\epsilon$ H\"older continuous for any $\epsilon\in (0, H+\alpha-1)$ almost surely, ending the proof.
\end{proof}

\begin{lemma}\label{lmm:groupforG}
Let $\{g_{\alpha}\}$ be the convolution group in \cite{liliu_prep}. In particular, for $\alpha> -1$
\begin{gather*}
g_{\alpha}=\displaystyle
\begin{cases}
\frac{\theta(t)}{\Gamma(\alpha)}t^{\alpha-1},& \alpha>0,\\
\delta(t), &\alpha=0,\\
\frac{1}{\Gamma(1+\alpha)}D\left(\theta(t)t^{\alpha}\right), & \alpha\in (-1, 0).
\end{cases}
\end{gather*} 
Here, $\theta(t)$ is the Heaviside step function while $D$ means the distributional derivative with respect to $t$. Let $\alpha_1\in (1-H, 1)$ and $\alpha_2+\alpha_1\in (1-H, 1)$.
Then,  it holds that $g_{\alpha_2}*G_{\alpha_1}=G_{\alpha_1+\alpha_2}$.
\end{lemma}
\begin{proof}
  It suffices to look at a continuous path of $B_H$. For such a path, we can
  mollify to $B_H^{\e}=B_H*\eta_{\e}$ where
  $\eta_{\epsilon}=\frac{1}{\epsilon}\eta(\frac{t}{\epsilon})$ with $\eta\in C_c^{\infty}(-\infty, 0)$, $0\le\eta\le1$ and $\int_{-\infty}^{\infty} \eta dt=1$. Then, $g_{\alpha_2}*(g_{\alpha_1}*\frac{d}{dt}B_{H}^{\e})=g_{\alpha_1+\alpha_2}*\frac{d}{dt}B_H^{\e}$ by \cite{liliu_prep}. Taking $\epsilon\to 0$ and using the H\"older continuity of $B_H$, we arrive at the conclusion.
\end{proof}

\section{Existence of the strong solutions}\label{sec:exisuniq}

For the discussion on existence of solutions of a class of SDEs driven by fractional Brownian motion, one may refer to \cite{no02,hairer05}. Our FSDEs are different from those studied in \cite{no02,hairer05}, as we have both the Caputo derivatives and fractional Brownian motions.

Mathematically, for fractional differential equations with Caputo derivative of order $\alpha\in (0,1)$, we only need to specify the initial value at one point $t=0$ (\cite{df02,diethelm10,liliu_prep}). For our fractional SDE, this is clear from Equation \eqref{eq:fsde2}. Intuitively, the system is activated at $t=0$ and one starts to count the memory effect from $t=0$. For better understanding, our model is the over-damped  case of the generalized Langevin equation, which is derived from  Kac-Zwanzig model (see \cite{gks04,kou08}). In the Kac-Zwanzig model, specifying the value at $t=0$ is enough. Hence, to make the FSDE solvable, we only need to specify the data at $t=0$.

We first define the so-called strong solution:
\begin{definition}\label{def:fsdestrongsoln}
Given a probability space $(\Omega, \mathcal{F}, P)$ and a random variable $x_0$ on this space, suppose $B_H$ is a fractional Brownian motion over this space, which may be coupled to $x_0$. A Strong solution of the fractional stochastic differential equation \eqref{eq:fsde1} with initial condition $x_0$ on the interval $[0, T)$ ($T>0$) is a process $x(t)$ that is continuous and adapted to the filtration $(\mathcal{G}_t)$ with $\mathcal{G}_t=\cap_{s>t} (\sigma(B_H(\tau), 0\le\tau\le s)\cup\sigma(x_0))$, $\forall t\in [0, T)$, satisfying 

(1) $P(x(0)=x_0)=1$. 

(2) With probability one, we have $\forall t\in [0, T)$, Equation \eqref{eq:fsde2} holds.
\end{definition}

It is standard to prove that the strong solution exists and is unique given the initial data since we have shown that process $G$ is continuous. We state the theorem and put the proof in the appendix for a reference:
\begin{theorem}\label{thm:existenceuniqueness}
Let $H>1/2$ and $\alpha\in(1-H, 1)$. Assume that $V'(\cdot)$ is Lipschitz continuous. Then, there exists a unique strong solution on $[0,\infty)$ to the FSDE \eqref{eq:fsde1} for a given fractional Brownian motion and initial distribution in the sense of Definition \ref{def:fsdestrongsoln}.
\end{theorem}

If $V'(x)$ is only locally Lipschitz, we probably need $V$ to be confining, or in other words, $\lim_{|x|\to\infty} V(x)=\infty$ and $e^{-\beta V(x)}\in L^1(\mathbb{R})$ for any $\beta>0$ for the global existence of the solution. We are not going to pursue this issue any further in this work.

\section{Asymptotic analysis}\label{sec:asym}

In this section, we discuss the ergodicity and convergence to equilibrium satisfying Gibbs measure. This is important for a physical system. When the force is linear, we show rigorously that the process is ergodic and converges algebraically to the Gibbs measure when the `fluctuation-dissipation theorem' is satisfied (see Theorem \ref{thm:linear}).  When the force is nonlinear, we believe this problem must be solved by rewriting the FSDE model into Markovian processes and we propose two possible approaches for this.

\subsection{Linear force case}
Consider that $V'(x)=kx$ for some $k>0$. By Theorem \ref{thm:existenceuniqueness}, the solution exists and is unique.  
In this section, we will show rigorously for the linear force case that our model indeed has physical meaning. 

For the convenience,  we introduce the function
\begin{gather}
e_{\alpha,k}(t)=E_{\alpha}(-k t^{\alpha}),
\end{gather}
where
\begin{gather}\label{eq:mlseries}
E_{\alpha}(z)=\sum_{n=0}^{\infty}\frac{z^n}{\Gamma(n\alpha+1)}
\end{gather}
is the Mittag-Leffler function. Note that $e_{\alpha,k}$ solves the equation 
\[
D_c^{\alpha}e_{\alpha,k}=-ke_{\alpha,k},~ e_{\alpha,k}(0)=1.
\]
 See for example \cite{mpg07}.

We can now state the result for asymptotic behaviors:
\begin{theorem}\label{thm:linear}
Let $V=\frac{1}{2}kx^2$. As $t\to\infty$, the solution $x(t)$ to the FSDE \eqref{eq:fsde1} converges in distribution to a normal distribution, i.e., $x(t)$ tends to a stationary Gaussian process: $x_{\infty}(t)$. The covariance $h(\tau)=\mathbb{E}(x_{\infty}(t)x_{\infty}(t+\tau))$ of this stationary process satisfies
\begin{gather}
\mathcal{F}(h(\tau))=\frac{2\Gamma(2H+1)\sin(H\pi)}{\Gamma(1-\alpha)}\frac{|\omega|^{1-2H}}{|(i\omega)^{\alpha}+k|^2},
\end{gather}
where $\mathcal{F}(\cdot)$ is the Fourier transform operator for tempered distributions.
If $\alpha=\alpha^*$ so that the `fluctuation-dissipation theorem' is satisfied, the covariance is given exactly by
\begin{gather}
h(\tau)=\frac{1}{k}e_{\alpha,k}(\tau).
\end{gather}
In particular,  $x_{\infty}(t)$ satisfies the Gibbs measure
\[
\mu(dx)\sim \exp\left(-\frac{1}{2}kx^2\right)dx.
\]
\end{theorem}

To prove this theorem, our approach is to find out the exact formulas for the solutions:
\begin{multline}\label{eq:solnlinear}
x(t)=x_0e_{\alpha,k}(t)+ \biggl( G(t)+\int_0^t G(t-s)\dot{e}_{\alpha,k}(s)\,ds \biggr)\\
=x_0e_{\alpha,k}(t)-\frac{C_H}{k}\int_0^t \dot{e}_{\alpha,k}(t-\tau)\,dB_H(\tau)
=: X_1 + X_2.
\end{multline}
Recall again that the dot means derivative on time.  We will analyze the asymptotic behaviors of this stochastic process to conclude our claim.

Before we give the proof, we remark that the variance of the first term is $\sim t^{-2\alpha}$ while the variance of the second term increases to the stationary variance with
rate $t^{2H-2-\alpha}$. The loss of the variance of the first term can be balanced by the gain of the second term only if
$-2\alpha=2H-2-\alpha$ or $\alpha=\alpha^*$.  If $\alpha$ is too small, then, the effect of initial data dampens slowly, or the
dissipation caused by viscosity is small, which cannot balance the fluctuation. If $\alpha$ is too big, then the effect of initial data
dampens too fast due to strong dissipation. Hence, the fluctuation-dissipation theorem must be satisfied to model a true
physical system so that there is balance.  We also remark that, as we have seen, even if there is no balance between fluctuation and dissipation, the whole process will still tends to a normal distribution, though it might not be the correct physical equilibrium.

We now move to the proof of Theorem \ref{thm:linear}. To do that, we prove several auxiliary lemmas. We first of all introduce a lemma regarding the behavior of $e_{\alpha,k}$:
\begin{lemma}\label{lmm:mlfund}
$e_{\alpha,k}$, the solution to the initial value problem
\[
D_c^{\alpha}e_{\alpha,k}=-ke_{\alpha,k},~ e_{\alpha,k}(0)=1,
\]
 is continuous on $[0,\infty)$ and smooth on $(0,\infty)$. As $t\to\infty$,
 \[
 e_{\alpha,k}= O(t^{-\alpha}).
 \]
 
The derivative is negative: $\dot{e}_{\alpha,k}(t)<0$. As $t\to 0^+$, $\dot{e}_{\alpha,k}(t)\sim Ct^{\alpha-1}$. Further, there exist $C_1>0, C_2>0$ such that for $t\ge 1$, 
\begin{gather}
C_1t^{-\alpha-1}\le |\dot{e}_{\alpha, k}(t)|\le C_2t^{-\alpha-1}.
\end{gather}
We have for $t>0$,
\begin{gather}\label{eq:ederiv}
\dot{e}_{\alpha,k}(t)=-k g_{\alpha}(t)- \frac{k}{\Gamma(\alpha)}\int_0^t(t-s)^{\alpha-1}\dot{e}_{\alpha,k}(s)\,ds.
\end{gather}
\end{lemma}
\begin{proof}
The fact that $e_{\alpha,k}$ is the solution to the IVP is well-known (\cite{mpg07}). 
Denoting the Heaviside step function as $\theta(t)$ and $g_{\alpha}=\frac{\theta(t)}{\Gamma(\alpha)}t^{\alpha-1}$. Using the group technique and the inverse formula introduced in \cite{liliu_prep}, we find that 
\[
\theta(t)e_{\alpha,k}=\theta(t)\left(1+g_{\alpha}*(-k\theta(t)e_{\alpha,k})\right).
\]
Taking the distributional derivative on both sides, we find that
\[
\theta(t)\dot{e}_{\alpha,k}=-k g_{\alpha}-k g_{\alpha}*(\theta(t)\dot{e}_{\alpha,k}).
\]
Since all distributions are locally integrable, the convolution can be written as Lebesgue integral and we have the equality \eqref{eq:ederiv}.

By the series expansion of Mittag-Leffler functions (Eq. \eqref{eq:mlseries}), we find the local behavior of $\dot{e}_{\alpha, k}$ near $t=0$. From the series expansion, it is seen that $e_{\alpha,k}$ is strictly decreasing on $(0,\infty)$.  The asymptotic behavior at $t\to\infty$, is obtained by Tauberian analysis (\cite{widder41}) using the Laplace transforms of $\dot{e}_{\alpha,k}(t)$ and $\ddot{e}_{\alpha,k}(s)$ (noting the Laplace transform $\mathcal{L}(\dot{e}_{\alpha,k})=-\frac{k}{s^{\alpha}+k}$), or the asymptotic behavior of Mittag-Leffler function directly.
\end{proof}

\begin{lemma}
The solution to the FSDE \eqref{eq:fsde1} with $V'(x)=kx$ is given by \eqref{eq:solnlinear}.
\end{lemma}

\begin{proof}
Since $G(t)$ is almost surely continuous, we just solve the equation for each continuous sample path of $G(t)$.  Let $T>0$. We set $\widetilde{G}(t)=G(T)$ when $t>T$ so that the Laplace transform of $\tilde{G}$ exists. Consider the equation 
\begin{gather*}
y(t)=x(0)-\frac{k}{\Gamma(\alpha)}\int_0^t(t-s)^{\alpha-1}y(s)\,ds+\widetilde{G}(t).
\end{gather*}
Take the Laplace transform (denoted by $\mathcal{L}$) on both sides.
Since $\mathcal{L}(t^{\alpha-1})=\Gamma(\alpha)s^{-\alpha}$, we find
\begin{gather*}
\mathcal{L}(y)=\frac{x_0s^{\alpha-1}}{s^{\alpha}+k}+\mathcal{L}(\widetilde{G})\Bigl(1-\frac{k}{s^{\gamma}+k}\Bigr).
\end{gather*}
We have by the Laplace transform of $e_{\alpha,k}$ that 
\[
y(t)=x_0e_{\alpha, k}(t) +\widetilde{G}(t)+\int_0^t \widetilde{G}(t-s)\dot{e}_{\alpha,k}(s)\,ds.
\]
Clearly,
\[
x(t)=y(t),~t\le T.
\]
Since $T$ is arbitrary, then, we have for any $t\ge 0$ that the unique solution $x(t)$ is given by
\begin{gather*}
x(t)=x_0e_{\alpha,k}(t)+ \biggl( G(t)+\int_0^t G(t-s)\dot{e}_{\alpha,k}(s)\,ds \biggr).
\end{gather*}

Now, we argue that $X_2$ in \eqref{eq:solnlinear} can also be written as
\begin{gather*}
X_2(t)=-\frac{C_H}{k}\int_0^t \dot{e}_{\alpha,k}(t-\tau)\,dB_H(\tau).
\end{gather*}
We first of all rewrite
\[
\int_0^t G(t-s)\dot{e}_{\alpha,k}(s)\,ds=\frac{C_H}{\Gamma(\alpha)}\int_0^t\int_0^{t-s}(t-s-\tau)^{\alpha-1}dB_H(\tau) \dot{e}_{\alpha,k}(s)\,ds .
\]
As in the proof of Lemma \ref{lmm:groupforG}, we may mollify the random path. Then, we can change the order of integration. Taking the mollifying parameter to zero, we get
\[
\int_0^t G(t-s)\dot{e}_{\alpha,k}(s)\,ds=\frac{C_H}{\Gamma(\alpha)}\int_0^t\int_0^{t-\tau}(t-s-\tau)^{\alpha-1}\dot{e}_{\alpha,k}(s)\,ds\,dB_H(\tau).
\]

By the identity for $\dot{e}_{\alpha,k}$ (Eq. \eqref{eq:ederiv}), we have
\[
\frac{1}{\Gamma(\alpha)}\int_0^{t-\tau}(t-s-\tau)^{\alpha-1}\dot{e}_{\alpha,k}(s)ds
=-g_{\alpha}(t-\tau)-\frac{1}{k}\dot{e}_{\alpha,k}(t-\tau).
\]
This then yields
\[
\int_0^t G(t-s)\dot{e}_{\alpha,k}(s)\,ds=-\frac{C_H}{\Gamma(\alpha)}\int_0^t
(t-\tau)^{\alpha-1}dB_H(\tau)
-\frac{C_H}{k}\int_0^t\dot{e}_{\alpha,k}(t-\tau)\,dB_H(\tau).
\]
This then shows the claim.
\end{proof}

$X_2$ is Gaussian process since $B_H$ is a Gaussian process. The mean of $X_2$ is clearly zero. We can investigate the variance to see its asymptotic behavior.
\begin{remark}
$X_2$ never converges in $L^p$ or almost surely, so we only consider convergence in distribution.
\end{remark}

For notational convenience, let us denote
\begin{gather}\label{eq:eqr}
r(t)=-\dot{e}_{\alpha,k}(t)\ge 0.
\end{gather}
By the isometry \eqref{eq:innerh}, we can compute that
\begin{multline}
\Sigma(t)=\var(X_2(t))=H(2H-1)\frac{C_H^2}{k^2}\int_0^t \int_0^{t}r(t-s)r(t-\tau)|s-\tau|^{2H-2}d\tau ds\\
=\frac{1}{k^2\Gamma(1-\alpha)}\int_0^t\int_0^{t}r(s)r(\tau)|\tau-s|^{2H-2}d\tau ds.
\end{multline}

\begin{lemma}\label{lmm:sigmaconv}
Let $\alpha\in (1-H, 1)$. $\Sigma=\lim_{t\to\infty}\Sigma(t)$ exists and there exist $C_1>0, C_2>0$ such that
\begin{gather}
C_1t^{2H-2-\alpha}<\Sigma-\Sigma(t)<C_2t^{2H-2-\alpha}.
\end{gather}
\end{lemma}

\begin{proof}

By Lemma \ref{lmm:mlfund} and \eqref{eq:eqr},
 $r$ is positive  and 
\begin{gather*}
\int_0^{\infty}r(t)\,dt=1.
\end{gather*}
By Lemma \ref{lmm:mlfund}, there exist $C_1>0, C_2>0$ such that for $t\ge 1$
\[
C_1t^{-\alpha-1}\le r\le C_2t^{-\alpha-1}.
\]
Then, that $\Sigma=\lim_{t\to\infty}\Sigma(t)$ exists is clear.

Consider the remainder $\Sigma-\Sigma(t)$, which is an integral over
the region $\mathbb{R}_{\ge0}^2\setminus [0,t]\times[0, t]$. Due to the symmetry, we have 
\begin{gather*}
k^2\Gamma(1-\alpha)(\Sigma-\Sigma(t))=2\int_t^{\infty}ds~r(s)\int_0^s r(\tau)(s-\tau)^{2H-2}d\tau.
\end{gather*}
Consider that $t$ is large and therefore $s\ge t>1$.  Below, the letter $C$ denotes a generic constant which is independent of $s$ and $t$ but the concrete value could change from line to line. Denote the inside of the above integral as
\begin{gather*}
J(s)=\int_0^s r(\tau)(s-\tau)^{2H-2}d\tau\le \int_0^1  r(\tau)(s-\tau)^{2H-2}d\tau
+\int_1^{s}C \tau^{-\alpha-1}|s-\tau|^{2H-2}d\tau.
\end{gather*}
The first term is controlled by $(s-1)^{2H-2}\int_0^1r(\tau)d\tau$. The second term is estimated as:
\begin{multline*}
Cs^{2H-2-\alpha}\left(\int_{1/s}^{1/2} z^{-1-\alpha}(1-z)^{2H-2}dz+\int_{1/2}^1z^{-1-\alpha}(1-z)^{2H-2} dz \right) \\
\le Cs^{2H-2-\alpha}\left(2^{2-2H}\frac{1}{\alpha}(s^{\alpha}-2^{\alpha})+\bar{C}\right)\le Cs^{2H-2},
\end{multline*}
where $\bar{C}=\int_{1/2}^1z^{-1-\alpha}(1-z)^{2H-2} dz$ independent of $s$. Hence by the asymptotic behavior of $r$,
\[
\Sigma-\Sigma(t)\le C\int_t^{\infty}|r(s)|s^{2H-2}du\le C t^{2H-2-\alpha}.
\]

For the other direction, we just note $J(s)\ge s^{2H-2}\int_0^1 |r(\tau)|\,d\tau$.
\end{proof}

\begin{proof}[Proof of Theorem \ref{thm:linear}]
By inspection of the solution \eqref{eq:solnlinear}, it is clear that $X_1\to 0$ almost surely and in $L^2$ as $t\to\infty$. We only have to focus on $X_2$. 

Since $X_2$ is a Gaussian process with mean zero, by Lemma \ref{lmm:sigmaconv}, $\var(X_2)$ converges and thus $X_2$ converges in distribution. We can now show the convergence of the covariance 
\begin{multline*}
h(\tau; t)=\mathbb{E}(X_2(t)X_2(t+\tau))
=\frac{1}{k^2\Gamma(1-\alpha)}\int_0^t \int_0^{t+\tau}r(t-u)r(t+\tau-v)|u-v|^{2H-2}dvdu\\
=\frac{1}{k^2\Gamma(1-\alpha)}\int_0^t\int_0^{t+\tau}r(u)r(v)|v-\tau-u|^{2H-2}dvdu.
\end{multline*}

Denote
\begin{gather}\label{eq:covarianceh}
h(\tau)=\frac{1}{k^2\Gamma(1-\alpha)}\int_0^{\infty}\int_0^{\infty}r(u)r(v)|v-\tau-u|^{2H-2}dvdu \ge 0.
\end{gather}
We argue that there exists $C>0$ depending on $H,\alpha, k$ such that $h(\tau)\le C$. To do this, we use $\tilde{r}(t)$ to represent the even extension of $r$. Then, we find 
\begin{multline*}
\int_0^{\infty}\int_0^{\infty}r(u)r(v)|v-\tau-u|^{2H-2}dvdu=
\int_{\tau}^{\infty}\int_0^{\infty}r(u-\tau)r(v)|v-u|^{2H-2}dvdu\\
\le \int_{0}^{\infty}\int_0^{\infty}\tilde{r}(u-\tau)r(v)|v-u|^{2H-2}dvdu
\le \frac{1}{H(2H-1)}\|\tilde{r}(\cdot-\tau)\|_{\mathscr{H}_0}
\|r\|_{\mathscr{H}_0}
\end{multline*}
The last inequality follows from the fact that \eqref{eq:innerh} gives the inner product. According to Lemma \ref{lmm:mlfund}, the asymptotic behavior of $r(t)=-\dot{e}_{\alpha, k}$ implies that 
\[
\|\tilde{r}(\cdot-\tau)\|_{1/H}^{1/H}\le 2\|r\|_{1/H}^{1/H}<\infty
\]
since $\alpha\in (1-H, 1)$. Lemma \ref{lmm:embeddingH} implies that $h(\tau)$ is bounded.

This means that the integral in \eqref{eq:covarianceh} converges as $t\to\infty$, for $-\infty<\tau<\infty$:
\[
h(\tau; t)\to h(\tau)=\frac{1}{k^2\Gamma(1-\alpha)}\int_0^{\infty}\int_0^{\infty}r(u)r(v)|v-\tau-u|^{2H-2}dvdu \le C.
\]
Consequently, $X_2$ converges in distribution to a stationary process. The limit process $x_{\infty}(t)$ has the covariance $h(\tau)$ to be bounded. 

Since $h(\tau)$ is bounded, it is a tempered distribution. The Fourier transform exists. The following formal computation can be justified by considering $h(\tau)e^{-\e \tau^2}$ and then taking the limit $\e\to 0$ under the topology of the tempered distribution.
\begin{multline*}
\int_{-\infty}^{\infty}e^{-i\omega \tau} \int_0^{\infty}\int_{\tau}^{\infty}r(u-\tau)r(v)|u-v|^{2H-2}du dv d\tau \\
=\int_0^{\infty}\int_{-\infty}^{\infty} \int_{-\infty}^u r(u-\tau) e^{-i\omega\tau}d\tau |u-v|^{2H-2}r(v) dudv.
\end{multline*}
The inner most integral turns out to be 
\begin{gather*}
e^{-i\omega u}\int_0^{\infty}e^{i\omega\tau}r(\tau)d\tau=I(-i\omega)e^{-i\omega u}, 
\end{gather*}
with 
\begin{gather*}
I(s)=\frac{k}{s^{\alpha}+k}.
\end{gather*} 
The whole thing turns out to be 
\begin{gather*}
I(-i\omega)I(i\omega)\int_{-\infty}^{\infty}e^{-i\omega z}|z|^{2H-2}dz
= \frac{k^2}{|(i\omega)^{\alpha}+k|^2}(2\Gamma(2H+1)\sin(H\pi))|\omega|^{1-2H}).
\end{gather*}
This shows the first claim.

If $\alpha=\alpha^*=2-2H$, we find that 
\[
\mathcal{F}(h(\tau))=\frac{2\sin(H\pi)|\omega|^{1-2H}}{|(i\omega)^{\alpha}+k|^2}.
\]

Recall that we have the identity 
\[
\int_0^{\infty}e^{-ts}E_{\alpha}(-k t^{\alpha})dt=\frac{s^{\alpha-1}}{s^{\alpha}+k}.
\]
It follows that 
\[
\int_{-\infty}^{\infty}e^{-i\omega t}E_{\alpha}(-k|t|^{\alpha})dt=2\frac{Re((i\omega)^{\alpha-1})(k+(-i\omega)^{\alpha})}{|k+(i\omega)^{\alpha}|^2}
=\frac{2k\sin(\alpha \pi/2)|\omega|^{\alpha-1}}{|k+(i\omega)|^2}.
\]
Hence, we find in this case 
\begin{gather*}
h(\tau)=\frac{1}{k}e_{\alpha, k}(\tau).
\end{gather*}

It follows that the final equilibrium is a normal distribution with variance 
\[
\Sigma=h(0)=\frac{1}{k}
\]
 and the last claim follows.
\end{proof}

\subsection{The general case}

We have proved that for linear regimes, when $\alpha=\alpha^*$ is considered, the distribution converges to the Gibbs measure with algebraic rate. The linear forcing case is special, but it shows that our model makes physical meaning.
For general forcing regimes with the `fluctuation-dissipation theorem' satisfied ($\alpha=\alpha^*$), proving the ergodicity and that the distribution converges to the Gibbs measure algebraically seems hard. We believe this problem can be solved by figuring out some Markovian representations. In the following, we propose two such possible Markovian embedding approaches that may be helpful for studying the asymptotic behavior. 

\subsubsection{Infinitely dimensional Ornstein-Uhlenbeck process with mixing}

If the kernel $\gamma(t)$ is the sum of finitely many exponentials, it is well known the GLE has a Markovian representation with a particular mixing (see \cite{op11} for the details) so that the Gibbs measure is an invariant measure. However, the result corresponding to a general kernel with fat tail is yet unknown.  In our FSDE, the kernel $\gamma(t)=\frac{\theta(t)}{\Gamma(1-\alpha)}t^{-\alpha}$ is of fat tail but it is completely monotone. A completely monotone function is the Laplace transform of a Radon measure on $[0,\infty)$ by the famous Bernstein theorem \cite{widder41}. In other words, the kernel $\gamma(\cdot)$ can be written as superpositions of infinitely many exponentials. Based on this observation, we can formally rewrite our FSDE model to an infinite-dimensional OU process with mixing. We hope the techniques in \cite{op11} may be generalize to this infinite OU process to discuss the ergodicity of our FSDE model. This seems beyond the scope of this paper and we leave the rigorous discussion to future.

To understand the idea, we first of all consider the deterministic equation 
\begin{gather}\label{eq:fode1}
D_c^{\alpha}x=\gamma(t)*(\theta(t)\dot{x})=x,\ \ x(0)=x_0.
\end{gather}
It is well-known that the solution of this equation is $x(t)=x_0E_{\alpha}(t^{\alpha})$, which is continuous on $[0,\infty)$ and smooth on $(0,\infty)$, and further $\dot{x}\ge 0$ \cite{liliu_prep}. 

The kernel $\gamma(t)$ is completely monotone and it can be written as
\begin{gather}
\gamma(t)=\int_0^{\infty}e^{-\lambda t}\rho(\lambda)\,d\lambda,~~
\rho(\lambda)=\frac{1}{B(\alpha, 1-\alpha)} \lambda^{\alpha-1}.
\end{gather}
Here $B(\cdot, \cdot)$ is the Beta function. We then decouple the fractional ODE \eqref{eq:fode1} as an infinitely dimensional Markovian process with a mixing effect:
\begin{gather}\label{eq:mixing}
\displaystyle
\begin{cases}
0=x(t)+\xi(t), ~~ t>0,& x(0+)=x_0, \\
\dot{\xi}_{\lambda}(t)=-\lambda\xi_{\lambda}(t)-\sqrt{\rho(\lambda)}\dot{x}(t),& \xi_{\lambda}(0)=0, \\
\xi(t)=\lim_{\epsilon\to 0}\int_0^{\infty}e^{-\lambda \epsilon}\sqrt{\rho}\xi_{\lambda}(t)\,d\lambda.
\end{cases}
\end{gather}
We solve the second equation in \eqref{eq:mixing} as
\begin{gather}
\xi_{\lambda}(t)=-\int_0^t\sqrt{\rho(\lambda)}e^{-\lambda(t-s)}\dot{x}(s)\,ds, 
\end{gather}
which implies that $\xi$ in the third equation is well-defined. 
Provided $\dot{x}\ge 0$, we switch the order of integration for $\xi$ and applying monotone convergence theorem,
\begin{multline}
\xi(t)=-\lim_{\epsilon\to 0}\int_0^{\infty}\int_0^t\rho(\lambda) e^{-\lambda(t-s+\epsilon)}\dot{x}(s)\,ds\,d\lambda\\
=-\lim_{\epsilon\to 0}\frac{1}{\Gamma(1-\alpha)}\int_0^t(t-s+\epsilon)^{-\alpha}\dot{x}(s)\,ds=-D_c^{\alpha}x(t).
\end{multline}
The equation $x=D_c^{\alpha}x,\ t>0$ then follows. This system then decouples the memory to a system of uncountable Markovian functions with the simple mixing given by the third equation in \eqref{eq:mixing}. 

\begin{remark}\label{rmk:subtlety}
Let us mention a subtlety of the system: it seems that the initial
value of $x$ is unimportant as one can reduce the system to
\begin{gather*}
\begin{split}
& \dot{\xi}_{\lambda}(t)=-\lambda\xi_{\lambda}(t)+\sqrt{\rho(\lambda)}\dot{\xi}(t),~ t>0.~~\xi_{\lambda}(0)=0.\\
& \xi(t)=\lim_{\epsilon\to 0}\int_0^{\infty}e^{-\lambda \epsilon}\sqrt{\rho(\lambda)}\xi_{\lambda}(t)\,d\lambda.
\end{split}
\end{gather*}
This seems to be solvable without considering $x_0$. Actually, this system is not well-posed. The reason is that the equation for $\xi_{\lambda}$ may not be valid at $t=0$ and $\lim_{t\to 0}\xi(t)\neq \xi(0)=0$. (In the original system, $\xi(0)=\xi(0+)=x(0+)$ is equivalent to $\lim_{t\to 0}D_c^{\alpha}x=0$.) We must know $\lim_{t\to 0}\xi(t)=\lim_{t\to 0}D_c^{\alpha}x$ to start the process, which is equivalent to assigning the initial value of $x$.
\end{remark}

Back to our FSDE \eqref{eq:fsde1}, the computation for the deterministic case then leads us to consider: 
\begin{gather}
\displaystyle
\begin{cases}
V'(x(t))=\xi(t),& t>0\\
\xi(t)=\lim_{\epsilon\to 0^+}\int_0^{\infty}\xi_{\lambda}(t)e^{-\epsilon\lambda}\rho(\lambda)^{1/2}\,d\lambda,&t>0\\
\dot{\xi}_{\lambda}(t)=-\lambda\xi_{\lambda}(t)-\sqrt{\rho(\lambda)}\dot{x}(t)+\sqrt{2\lambda}\dot{W}_{\lambda}(t). & 
\end{cases}
\end{gather}
 Here we assume $\xi_{\alpha}(0)$'s are i.i.d, normal with variance $1$. This is a random system of differential algebraic equations (DAE), and clearly Markovian. The issue is that we have an uncountable-dimensional stochastic process driven by an uncountable-dimensional Wiener process (normal Brownian motion).  
 
With the random noise, we may not be able justify the computation as we did for the deterministic cases. However, a formal computation may still be illustrating, through which we argue that this DAE system is equivalent to our FSDE. By solving $\xi_{\lambda}$ formally, we have
\begin{multline}\label{eq:randomxi}
\xi(t)=\lim_{\epsilon\to 0^+}\left(\int_{[0,\infty)}\xi_{\lambda}(0) \sqrt{\rho} e^{-\lambda(t+\epsilon)}d\lambda
+\int_{[0,\infty)}\int_0^t\sqrt{2\lambda\rho} e^{-\lambda(t-s+\epsilon)}dW_{\lambda}(s)d\lambda\right)\\
-\lim_{\epsilon\to 0} \int_{[0,\infty)}\int_0^t \rho(\lambda)e^{-\lambda(t-s+\epsilon)}\dot{x}(s)\,ds d\lambda
=:R(t)+K(t).
\end{multline}
In the case $t>0, \tau\ge 0$, we have 
\begin{multline}
\mathbb{E}(R(t)R(t+\tau))=\int_{[0,\infty)}\rho e^{-\lambda(2t+\tau)} \var(\xi_{0})\,d\lambda
+\int_{[0,\infty)}\int_0^t2\lambda\rho(\lambda) e^{-\lambda(2t+\tau-2s)}dsd\lambda\\
=\gamma(\tau+2t)+\gamma(\tau)-\gamma(\tau+2t)=\gamma(\tau).
\end{multline}
Of course, the change of order of integration and expectation is not justified rigorously, but the computation is still interesting. Since both $R(t)$ and $C_H\dot{B}_H$ are Gaussian process and they have the same covariance, we can then identify them.

For the term $K(t)$ in \eqref{eq:randomxi}, since $\rho(\lambda) e^{-\epsilon\lambda}\in L^1[0,\infty)$, we may change the order of integration and $K(t)=-D_c^{\alpha}x$ for $t>0$. Hence, 
\begin{gather}
\xi=-D^{\alpha}x+R(t), t>0.
\end{gather}
This then formally verifies that FSDE \eqref{eq:fsde1} can be obtained from the Markovian DAE system.

The same subtlety in Remark \ref{rmk:subtlety} appears here. $\xi(0)\neq -D^{\alpha}x|_{t=0}+R(0)$, which allows us to specify the initial condition $x_0$. 

Since the Gibbs measures for the GLE with the kernel to be finitely exponentials are invariant measures \cite{op11}, we think it is promising to show that Gibbs measures are the final equilibrium measures for our model. The discussion here provides a possible framework for the study of general $V(x)$. To study the stochastic DAE system, one may have to put some structure in the space of infinite-dimensional Gaussian process, and then somehow figure out that the Gibbs measure for the whole system is an invariant measure. This will then be left for future.
  
\subsubsection{A heat bath model}

In this subsection, we summarize the heat bath model  proposed in \cite{jp97,rt02} for the generalized Langevin equation. The key point is that one can consider the whole dynamics of the particle together with the heat bath, which is Markovian. If one integrates out the degrees of freedom for the heat bath, one obtains the GLE.  The whole heat bath model is the continuous version of the Kac-Zwanzig model mentioned in \cite{zwanzig73,gks04,kou08}.
We think this heat bath model may be another promising direction to study the ergodicity and the asymptotic behavior of our FSDE model.   Formally, if one takes the $m\to 0$ limit for the special kernel $\gamma(t)\propto |t|^{-\alpha}$ (the discussion in \cite{jp97,rt02} is not applicable to this kernel though), our FSDE can be obtained. This limit for the classical Langevin equation (Eq. \eqref{eq:langevin}) is called the Smoluchowski-Kramers approximation \cite{freidlin04} and the limit for generalized Langevin equation has not been studied yet to our best knowledge.  We will summarize the formulation here briefly and then give a brief discussion to connect it with our FSDE model.

Assume that the particle is put in a heat bath modeled by infinitely many free phonons and the corresponding scalar field $\varphi$ is given by the massless Klein-Gordon equation (which is a wave equation), 
\begin{gather}
(-\partial_{t}^2+\Delta)\varphi=0.
\end{gather}
The Lagrangian density is $
\mathcal{L}=-\frac{1}{2}\partial^{\mu}\varphi\partial_{\mu}\varphi, $
where $\mu$ goes over the time-spatial coordinate in relativity, and the Hamiltonian is
\begin{gather}
\mathcal{H}_h=\frac{1}{2}\int_{\mathbb{R}^n} (|\nabla\varphi|^2+|\pi|^2)dx,
\end{gather}
where $\pi=\partial_t\varphi$ should be regarded as a new variable.

This Hamiltonian motivates that the correct space for the heat bath is 
\[
\mathscr{V}=H^1(\mathbb{R}^n)\otimes L^2(\mathbb{R}^n)
\]
 with the inner product given by
\begin{gather}\label{eq:innerHbath}
\langle f, g\rangle=\int_{\mathbb{R}^n} (\nabla f_1\cdot\nabla g_1+f_2 g_2)\,dx, \forall f=(f_1, f_2)\in\mathscr{V}, g=(g_1, g_2)\in\mathscr{V} .
\end{gather}
Note that Gaussian measures can be constructed over this Hilbert space. $\forall f, g\in \mathscr{V}$ and $\xi$ is an $\mathscr{V}$-valued random variable satisfying a Gaussian measure $\mu_{\phi_0}^{\beta}$ indexed by $\phi_0\in\mathscr{V}$ and $\beta>0$, then, 
\begin{gather}\label{eq:gaussianident}
\mathbb{E}(\langle f, \xi-\phi_0\rangle \langle \xi-\phi_0, g\rangle)=\beta^{-1}\langle f, g\rangle.
\end{gather} 

The coupling between the particle and the heat bath is given by 
\begin{gather*}
\mathcal{H}_I=\int_{\mathbb{R}^n} \varphi(x)\rho(q-x)\,dx=\int_{\mathbb{R}^n} \varphi(x)\rho(x-q)\,dx, 
\end{gather*}
where $\rho$ is a radially symmetric function which can be understood as the coupling strength.
In literature \cite{jp97,rt02}, $\rho$ is assumed to be in $L^2$, so that the coupling strength is finite and can be approximated by the dipole expansion:
\begin{gather}
\mathcal{H}_I=\int_{\mathbb{R}^n}  \rho q\cdot\nabla\varphi\,dx+\frac{q^2}{2}\int_{\mathbb{R}^n} |\rho^2|\,dx.
\end{gather}
The second term is some correction added to make the model clean so that the GLE can be derived from this model.

The total Hamiltonian that describes the coupling between the particle and the heat bath is given by 
\begin{multline}
\mathcal{H}=\frac{1}{2m}p^2+V(q)+\frac{1}{2}\int_{\mathbb{R}^n} (|\pi|^2+|\nabla\varphi|^2)\,dx+\int_{\mathbb{R}^n} \rho q\cdot\nabla\varphi \,dx+\frac{q^2}{2}\int_{\mathbb{R}^n} |\rho^2|\,dx\\
=\frac{1}{2m}p^2+V(q)+\frac{1}{2}\int_{\mathbb{R}^n}|\nabla\varphi+q\rho|^2+|\pi^2|\,dx .
\end{multline}
where $\lim_{|q|\to\infty}V(q)=\infty$ and $\exp(-\beta V(\cdot))\in L^1(\mathbb{R}^n)$ for any $\beta>0$.

With this coupling, the authors in \cite{jp97,rt02} showed that the particle satisfies the generalized Langevin equation obeying the `dissipation-fluctuation theorem' provided the initial data satisfy a certain Gaussian measure. The GLE for $n=1$ case is written as 
\begin{gather*}
\dot{q}=v, ~~
m\dot{v}=-V'(q)-\int_0^t\gamma(t-s)\dot{q}(s)\,ds+R(t),\\
\gamma(t)=\int_{\mathbb{R}}|\hat{\rho}|^2 e^{ikt} dk,~~\mathbb{E}(R(t)R(s))=\gamma(|t-s|).
\end{gather*}
With this result, the authors conclude the following:
\begin{proposition}
Suppose $R(t)$ is a $1D$ stationary Gaussian process with mean zero and
\begin{gather}
\mathbb{E}(R(t)R(s))=\gamma(|t-s|).
\end{gather} 
 If $\gamma$ is the Fourier transform of an $L^1(\mathbb{R})$ even nonnegative function, then there exists a coupling between $q(0)=q_0$ and $R(t)$ so that the equation 
\begin{gather}
\dot{q}=v, ~~
m\dot{v}=-V'(q)-\int_0^t\gamma(t-s)\dot{q}(s)\,ds+R(t)
\end{gather}
admits the Gibbs measure 
\begin{gather}\label{eq:particlegibbs}
\mu(dqdv)\propto \exp\left(-\frac{mv^2}{2}-V(q)\right)\,dqdv,
\end{gather}
as the invariant measure.

For any initial distribution $\mu^0$ that is absolutely continuous with respect to $\mu$ and any coupling between $q_0$ and $R(t)$,  $\mu^t$ converges weakly to the Gibbs measure $\mu$.
\end{proposition}

Our FSDE model is similar to the problems studied in  \cite{jp97,rt02}, except that  $\gamma(t)\propto |t|^{-\alpha}$ and $m=0$. Note that the kernel $|t|^{-\alpha}$ is not the Fourier transform of an $L^1$ kernel. One can therefore mollify $\gamma$ by
\begin{gather}
\gamma_{\e}(t)=\eta_{\e}*\gamma(t), 
\end{gather}
so that $\gamma_{\e}$ is the Fourier transform of an $L^1$ kernel. One can then study the GLE with kernel $\gamma_{\e}$. If final equilibrium is preserved with $\epsilon\to 0$ limit, then the Gibbs measure is the equilibrium measure for the GLE
with kernel $|t|^{-\alpha}$. Then, formally, the Smoluchowski-Kramers approximation $m\to 0$ limit (if valid) yields
that the Gibbs measure proportional to $\exp(-V(q))$ is the final equilibrium measure of our FSDE \eqref{eq:fsde2}.  This provides another possible framework for general potential $V(x)$ and we leave the rigorous study for future.

\begin{remark}
 The Smoluchowski-Kramers approximation ($m\to 0$ limit) for the usual Langevin equations has been discussed in \cite{freidlin04}. However, for the generalized Langevin equation, the limit $m\to 0$ is subtle. The limit equation for a general kernel $\gamma$ may not
be a good initial value problem. The initial value problem
\[
\int_0^t\gamma(t-s)\dot{q}(s)\,ds=-V'(q)+R(t), ~~ q(0)=q_0
\]
generally admits no continuous solution if $\gamma(t)$ is bounded. Hence, the possible approach is to show first that convergence to Gibbs measure is valid for the GLE when $\gamma(t)\propto |t|^{-\alpha}$ and then show the $m\to 0$ limit can pass to the final equilibrium measures.
\end{remark}

\section*{Acknowledgements}

The work of J.-G Liu is partially supported by KI-Net NSF RNMS11-07444 and NSF DMS-1514826. The work of J. Lu is supported in part  by National Science Foundation under grant DMS-1454939. J. Lu would also like to thank Eric Vanden-Eijnden for helpful discussions. 

\appendix

\section{Proof of Theorem \ref{thm:existenceuniqueness}}

\begin{proof}
We just consider a sample point $x_0$ and a sample path $G$ with $G$ being continuous. We then construct a path that satisfies the integral equation given this sample initial data.

By Proposition \ref{pro:contofg}, $G(t)$ is continuous. Consider the sequence given by 
\begin{gather*}
x^{(0)}=x_0, 
\end{gather*} 
and $x^{(n)}, n\ge 1$ is given by
\begin{gather*}
x^{(n)}(t)=x_0-\frac{1}{\Gamma(\alpha)}\int_0^t(t-s)^{\alpha-1}V'(x^{(n-1)}(s))\,ds+G(t).
\end{gather*}

Assume $L$ is a Lipschitz constant for $V'(\cdot)$. Introducing $g_{\gamma}=\frac{\theta(t)}{\Gamma(\gamma)}t^{\gamma-1}$, we find that $\{g_{\gamma}\}_{\gamma>0}$ forms a convolution semigroup (Lemma \ref{lmm:groupforG}). We define 
\[
e^n=x^{(n)}-x^{(n-1)}.
\]
Explicit formula tells us that
\begin{gather*}
e^1=-V'(x_0)g_{\alpha+1}+G(t),
\end{gather*}
and that
\begin{gather*}
|e^n|=|-g_{\alpha}*(V'(x^{n-1})-V'(x^{n-2}))|\le L g_{\alpha}*|e^{n-1}|,\ \ n\ge 2.
\end{gather*}

Hence, 
\begin{gather*}
|e^n|\le L^{n-1} g_{(n-1)\alpha}*|e^1|.
\end{gather*}
Direct computation shows that $\sup_{0\le t\le T}g_{(n-1)\alpha}*|e^1|$ decays exponentially in $n$. Hence, $\sum_n |e^n|$
converges. It follows that $\sum_n e^n$ converges uniformly on any interval $[0, T]$ with $T\in (0,\infty)$. The limit is also a continuous function. It turns out that the limit satisfies the integral equation.

For the uniqueness, assume that both $x(t)$ and $y(t)$ are solutions. Then, we take a sample where both $x(t)$ and $y(t)$ are continuous. For this sample, $\forall t>0$,
\[
|x(t)-y(t)|=\frac{1}{\Gamma(\alpha)}\left|\int_0^t(t-s)^{\alpha-1}(V'(x(s))-V'(y(s)))\right|ds
\le L (g_{\alpha}*|x-y|)(t).
\]

Applying this inequality iteratively and using the semi-group property of $g_{\gamma}$, we find 
\[
|x-y|(t)\le L^n g_{n\alpha}*|x-y|.
\]
Fixing $T>0$, the right hand side goes to zero uniformly on $[0, T]$. Then, we find that $x=y$ on $[0, T]$ for this sample path. Since both solutions are continuous almost surely, then $x=y$ on $[0, T]$ almost surely. By the arbitrariness of $T$, $x=y$ almost surely. The uniqueness then is shown. This then completes the proof of the theorem.
\end{proof}

\bibliographystyle{unsrt}
\bibliography{probsde}
\end{document}